\documentclass[1p,number,sort&compress]{elsarticle}
\usepackage{amsmath, mathrsfs, amssymb, enumerate, txfonts, color}
 
\numberwithin{equation}{section}
\allowdisplaybreaks

\newcommand{\wconv}{\rightharpoonup}
\newcommand{\wsconv}{\mathrel{\vbox{\offinterlineskip\ialign{\hfil##\hfil\cr $\hspace{-0.4ex}\textnormal{\scriptsize{*}}$\cr \noalign{\kern-0.6ex}  $\rightharpoonup$\cr}}}}
\newcommand*\diff{\mathop{}\!\mathrm{d}}
\newcommand*\graph{\mathop{}\!\mathrm{graph}}
\newcommand{\R}{\mathbb{R}}
\newcommand{\N}{\mathbb{N}}
\newcommand{\T}{\tilde{T}}
\newcommand{\wlim}{\mathop{}\!\mathrm{w\!-\!\overline{lim}}\;\!}

\newtheorem{theorem}{Theorem}[section]
\newtheorem{lemma}[theorem]{Lemma}

\newdefinition{definition}[theorem]{Definition}
\newdefinition{remark}[theorem]{Remark}
\newdefinition{example}[theorem]{Example}
\newproof{proof}{Proof}
\begin{document}
\begin{frontmatter}
\title{Existence of strong solutions for the Oldroyd model with multivalued right-hand side\footnote{This work is part of project A8 within Collaborative Research Center 910
``Control of self-organizing nonlinear systems: Theoretical methods and concepts of application'' that is supported by Deutsche Forschungsgemeinschaft.
}}
\author{Andr\'e Eikmeier}
\ead{eikmeier@math.tu-berlin.de}
\address{Technische Universit{\"a}t Berlin, Institut f\"ur Mathematik, Stra{\ss}e des 17.~Juni 136, 10623 Berlin, Germany}

\begin{abstract}
The initial value problem for a coupled system is studied. The system consists of a differential inclusion and a differential equation and models the fluid flow of a viscoelastic fluid of Oldroyd type. The set-valued right-hand side of the differential inclusion satisfies certain measurability, continuity and growth conditions. The local existence (and global existence for small data) of a strong solution to the coupled system is shown using a generalisation of Kakutani's fixed-point theorem  and applying results from the single-valued case.
\end{abstract}

\end{frontmatter}
\section{Introduction} \label{introduction}
\subsection{Problem statement}

\noindent Let $T>0$ and let $\Omega\subset \R^3$ be open, bounded, and connected with $\partial\Omega\in \mathscr{C}^{2,\mu}$, $0<\mu<1$. We consider the Oldroyd model for a viscoelastic fluid in the multivalued version
\begin{equation*} 
  \left\{\begin{aligned}
    \text{Re} \left(\partial_t u + (u\cdot\nabla) u\right) -(1-\alpha)\Delta u + \nabla p &\in \nabla\cdot \tau+F(\cdot ,u) \phantom{=02\alpha D(u)\tau_0f}\hspace{-1.5cm} \text{in }\Omega\times (0,T),\\
    \nabla \cdot u &= 0 \phantom{\in 2\alpha D(u)\tau_0\nabla\cdot \tau+fF(\cdot,u)}\hspace{-1.5cm} \text{in } \Omega\times(0,T),\\
    \text{We}\left( \partial_t \tau + (u\cdot\nabla)\tau + g_a(\tau,\nabla u)\right) + \tau &= 2\alpha D(u) \phantom{\in 0\tau_0\nabla\cdot \tau+fF(\cdot,u)}\hspace{-1.5cm} \text{in } \Omega\times(0,T),\\
    u &=0 \phantom{\in 2\alpha D(u)\tau_0\nabla\cdot \tau+fF(\cdot,u)}\hspace{-1.5cm} \text{on } \partial\Omega\times(0,T),\\
    u(\cdot,0)=u_0, \quad\quad \tau(\cdot,0)&=\tau_0 \phantom{\in 2\alpha D(u)0\nabla\cdot \tau+fF(\cdot,u)}\hspace{-1.5cm} \text{in } \Omega,
  \end{aligned}\right.
\end{equation*}
where $u\colon \overline{\Omega}\times [0,T]\to \R^3$ describes the velocity of the fluid, $\tau\colon \overline{\Omega}\times [0,T]\to \R^{3\times 3}$ describes the stress tensor of the fluid, and $p\colon \overline{\Omega}\times [0,T]\to \R$ describes the pressure in the fluid. Re is the Reynolds number, describing the relation between the inertial and viscous forces in the fluid, We is the Weissenberg number, describing the influence of elasticity on the fluid flow, and the parameter $\alpha\in (0,1)$ describes the influence of Newtonian viscosity on the fluid flow. Further, $a\in [-1,1]$ is a model parameter, and $g_a$ is given by
\begin{equation*}
  g_a(\tau,\nabla u) = \tau W(u) - W(u)\tau - a(D(u)\tau + \tau D(u))
\end{equation*}
with
\begin{equation*}
  D(u)=\frac{1}{2}(\nabla u + \nabla u^\top), \quad W(u)=\frac{1}{2}(\nabla u - \nabla u^\top).
\end{equation*}
The multivalued function $F$ fulfils certain measurability, continuity, and growth conditions that are given in detail in Section \ref{main_assumptions}. Finally, $u_0$ and $\tau_0$ are the given initial conditions.

Multivalued differential equations are used to, e.g., model feedback control problems. In this case, we can rewrite the first inclusion by introducing the single-valued right-hand side $f$ with
\begin{equation*}
  \begin{aligned}
    \text{Re} \left(\partial_t u + (u\cdot\nabla) u\right) -(1-\alpha)\Delta u + \nabla p &= \nabla\cdot \tau+f \phantom{\in F(\cdot,u)} \text{in }\Omega\times (0,T),\\
    f&\in F(\cdot,u) \phantom{= \nabla\cdot \tau+f } \text{in }\Omega\times (0,T),\\
  \end{aligned}
\end{equation*}
so we can consider $f$ as the control acting on the velocity $u$ and $F$ as the set of admissible controls, which in the case of a feedback control problem depends on the state $u$.

As already mentioned, the differential equation problem considered in this work relies on the so-called Oldroyd model for viscoelastic fluids, i.e., fluids that do not only show the common viscous behaviour, but also elastic behaviour. This model is used, e.g., for blood flow (see, e.g., Bilgi and Atal\i k~\cite{BilgiAtalik}, Bodn\'ar, Sequeira, and Prosi~\cite{BodSeqPro}, Smith and Sequeira~\cite{SmithSequeira}) or for polymer solutions, which appear, e.g., in microfluidic devices due to an enhanced mixing and heat transfer compared to Newtonian fluids (see, e.g., Arratia et al.~\cite{ArrThoDiGo}, Lund et al.~\cite{LundBrownEtAl}, Thomases and Shelley~\cite{ThomasesShelley}), or in applications containing drop formation of the fluid such as the prilling process or ink-jet printing (see, e.g., Alsharif and Uddin~\cite{AlsharifUddin}, Davidson, Harvie, and Cooper~\cite{DavHarCoo}).

\subsection{Literature overview}

This work is a combination of Fern\'andez-Cara, Guill\'en, and Ortega~\cite{FCGO02}, where the single-valued version of the problem above was considered, and Eikmeier and Emmrich \cite{EikEmm20}, where similar methods as in this article were used to prove existence of solutions to a multivalued differential equation problem with nonlocality in time.

The Oldroyd model is one of the most popular models for viscoelastic fluids and has risen a lot of attention in the last decades. Besides Fern\'andez-Cara, Guill\'en, and Ortega~\cite{FCGO02}, local existence and uniqueness of strong solutions has been proven in Guillop\'e and Saut~\cite{GuillopeSaut}, for a Besov space setting in Chemin and Masmoudi~\cite{CheminMasmoudi}, and, for a more general model, in Renardy~\cite{Renardy90}. As the Oldroyd model incorporates the usual Newtonian model, i.e., the Navier--Stokes equation, as a special case, global existence and uniqueness of weak solutions (in the three-dimensional case and for general initial data) have not yet been proven. For the special case of the Jeffreys model corresponding to $a=0$, global existence of weak solutions for general initial data was shown in Lions and Masmoudi~\cite{LionsMasmoudi}. For a more general constitutive equation including the Oldroyd model, but with stronger assumptions on the dissipation term in the balance of momentum, global existence of weak solutions for general initial data has been proven by Bejaoui and Majdoub~\cite{BejaouiMajdoub}.

Multivalued differential equations (or differential inclusions) have been studied by several authors as well, see, e.g., Aubin and Cellina~\cite{AubinCellina}, Aubin and Frankowska~\cite{AubinFrankowska}, or Deimling~\cite{Deimling} for basic results including theory from set-valued analysis. Extensions of the results shown in Deimling~\cite{Deimling} can be found in O'Regan~\cite{ORegan}. Existence of global solutions to multivalued differential equation with a linear parabolic principal part and a relaxed one-sided Lipschitz nonlinear set-valued operator is, e.g., considered in Beyn, Emmrich, and Rieger~\cite{BeynEmmRie}.

Multivalued differential equations in the context of the Oldroyd model have only been considered, up to the knowledge of the author, in Obukhovski\u{\i}, Zecca, and Zvyagin~\cite{ObZecZvy} for the special case of the Jeffreys model. Existence of solutions is shown via the topological degree theory.
Results for other viscoelastic models like the Voigt model can be found in, e.g., Gori et al.~\cite{Gori_etal} and in Zvyagin and Kuzmin~\cite{ZvyaginKuzmin}.

\subsection{Organisation of the paper}

In Section~2, we introduce the basic notation and repeat some results from set-valued analysis that we use in this work. In Section~3, we state the assumptions on the set-valued right-hand side $F$ and a preliminary result needed for the following proofs of the main results. The first of these results about the local existence is then stated and proven in Section~4. Finally, in Section~5, we state and prove the global existence of solutions for small data.

\section{Basic notation and introduction to set-valued analysis}

Given a Banach space $X$, we denote its dual by $X^*$, the norm in $X$ by $\Vert \cdot\Vert_X$, the standard norm in $X^*$ by $\Vert \cdot\Vert_{X^*}$ and the duality pairing by $\langle\cdot,\cdot\rangle$. In the case of a Hilbert space $X$, we denote the inner product by $(\cdot,\cdot)$.

Let $\Omega\subset \R^d$, $d\in \N$, be Lebesgue measurable and let $1\leq p\leq \infty$. We denote the usual Lebesgue spaces by $L^p(\Omega)$, equipped with the standard norm. In the case $p<\infty$, the dual space of $L^p(\Omega)$ is given by $L^{p'}(\Omega)$ with the conjugate $p'=p/(p-1)$ for $p>1$ and $p'=\infty$ for $p=1$. Analogously, for a real, reflexive, separable Banach space $X$ and for $T>0$, we denote the usual Bochner--Lebesgue spaces by $L^p(0,T;X)$, equipped with the standard norm. Again, in the case $p<\infty$, the dual space of $L^p(0,T;X)$ is given by $L^{p'}(0,T;X^*)$. The duality pairing in this case is given by
\begin{equation*}
  \langle g,f\rangle = \int_0^T \langle g(t),f(t)\rangle\diff t,
\end{equation*}
see, e.g., Diestel and Uhl~\cite[Theorem~1 on p.~98, Corollary~13 on p.~76, Theorem~1 on p.~79]{DiestelUhl}. 

By $W^{k,p}(\Omega)$, $k\in\N$, we denote the usual Sobolev space of $k$-times weakly differentiable functions $u\in L^p(\Omega)$ with $D^\beta u\in L^p(\Omega)$, where $\beta \in \N^d$ is a multiindex of order $\vert\beta\vert\leq k$. The spaces are again equipped with the standard norm. By $W_0^{k,p}(\Omega)$, $p>1$, we denote the space of all functions $u\in W^{k,p}(\Omega)$ with $u=0$ on $\partial\Omega$, also equipped with the standard norm. Similarly, we denote by $W^{1,p}(0,T;X)$ the space of weakly differentiable functions $u\in L^p(0,T;X)$ with $u'\in L^p(0,T;X)$, equipped with the standard norm. We have the continuous embedding $W^{1,1}(0,T;X) \subset \mathscr{C}([0,T];X)$, where $\mathscr{C}([0,T];X)$ denotes the space of all functions on $[0,T]$ with values in $X$ that are continuous, see, e.g., Roub\'i\v{c}ek~\cite[Lemma 7.1]{Roubicek}. A function $u\in W^{1,1}(0,T;X)$ is even almost everywhere equal to a function $u\in \mathscr{AC}([0,T];X)$, i.e., a function on $[0,T]$ with values in $X$ that is absolutely continuous, see, e.g., Br\'ezis~\cite[Theorem 8.2]{Brezis}. The space of all functions on $[0,T]$ with values in $X$ that are continuously differentiable is denoted by $\mathscr{C}^1([0,T];X)$, and the space of all functions on $[0,T]$ that are continuous with respect to the weak topology in $X$ is denoted by $\mathscr{C}_w([0,T];X)$.

Next, we introduce a few definitions from set-valued analysis. Let $(\Omega, \Sigma)$ be a measurable space and let $X$ be a complete separable metric space. We denote the Lebesgue $\sigma$-algebra on the interval $[a,b]\subset \mathbb{R}$ by $\mathcal{L}([a,b])$ and the Borel $\sigma$-algebra on $X$ by $\mathcal{B}(X)$. Further, we denote the set of all nonempty and closed subsets $U\subset X$ by $\mathcal{P}_{f}(X)$, the set of all nonempty and convex subsets $U\subset X$ by $\mathcal{P}_{c}(X)$, and the set of all nonempty, closed, and convex subsets $U\subset X$ by $\mathcal{P}_{fc}(X)$.

Let $F\colon \Omega\to 2^X\setminus \{\emptyset\}$ be a set-valued function. We define the pointwise supremum
\begin{equation*}
  \vert F(\omega)\vert:=\sup \left\{\Vert x\Vert_X\mid x\in F(\omega)\right\}\!,\quad \omega\in \Omega,
\end{equation*}
and the graph of $F$
\begin{equation*}
  \graph(F)=\left\{ (\omega,x)\in \Omega\times X \mid x\in F(\omega)\right\}\!.
\end{equation*}
We call a function $F\colon \Omega\to \mathcal{P}_f(X)$ measurable if the preimage of each open set is measurable, i.e.,
\begin{equation*}
  F^{-1}(U):=\left\lbrace \omega\in \Omega\mid F(\omega)\cap U \neq \emptyset\right\rbrace \in \Sigma
\end{equation*}
for every open $U\subset X$. For equivalent definitions, see, e.g., Denkowski, Mig\'orski, and Papageorgiou~\cite[Theorem 4.3.4]{DenMigPapa}. For every measurable set-valued function, there exists a measurable selection, i.e., a $\Sigma$-$\mathcal{B}(X)$-measurable function $f\colon \Omega\to X$ with $f(\omega)\in F(\omega)$ for all $\omega\in \Omega$, see, e.g., Aubin and Frankowska~\cite[Theorem 8.1.3]{AubinFrankowska}.

Now, let $(\Omega, \Sigma, \mu)$ be a complete $\sigma$-finite measure space, let $X$ be a separable Banach space, let $F\colon \Omega \to \mathcal{P}_{f}(X)$ be a set-valued function, and let $p\in [1,\infty)$. By $\mathcal{F}^p$, we denote the set of all $p$-integrable selections of $F$, i.e.,
\begin{equation*}
  \mathcal{F}^p:=\left\{ f\in L^p(\Omega;X,\mu) \mid f(\omega)\in F(\omega)\ \text{a.e. in}\ \Omega\right\}\!,
\end{equation*}
where $L^p(\Omega;X,\mu)$ is the space of Bochner measurable, $p$-integrable functions with respect to $\mu$.\footnote{If $X$ is a separable Banach space, the Bochner measurability of $f$ coincides with the $\Sigma$-$\mathcal{B}(X)$-measurability, see, e.g., Amann and Escher~\cite[Chapter X, Theorem 1.4]{AmannEscher}, Denkowski, Mig\'orski, and Papageorgiou~\cite[Corollary~3.10.5]{DenMigPapa}, or Papageorgiou and Winkert~\cite[Theorem 4.2.4]{PapaWin}} If there exists a nonnegative function $m\in L^p(\Omega;\mathbb{R},\mu)$ such that $F(\omega)\subset m(\omega)\;\! B_X$ for $\mu$-almost all $\omega\in \Omega$, where $B_X$ denotes the open unit ball in $X$, we call $F$ integrably bounded. In this case, Lebesgue's theorem of dominated convergence implies that every measurable selection of $F$ is an element of $\mathcal{F}^p$. The integral of a set-valued function $F$ is defined by
\begin{equation*}
  \int_\Omega F \diff \mu := \left\{ \int_\Omega f \diff \mu \mid f\in \mathcal{F}^1\right\}\!.
\end{equation*}
Important properties of this integral can be found, e.g., in Aubin and Frankowska~\cite[Chapter~8.6]{AubinFrankowska}.

Now, let $F$ have an additional argument, i.e., $F\colon \Omega \times X \to \mathcal{P}_f(X)$, and let $v\colon \Omega \to X$. By $\mathcal{F}^p(v)$, we denote the set of all $p$-integrable selections of the mapping $\omega \mapsto F(\omega, v(\omega))$, i.e.,
\begin{equation*}
  \mathcal{F}^p(v):=\left\{ f\in L^p(\Omega;X,\mu) \mid f(\omega)\in F(\omega, v(\omega))\ \text{a.e. in}\ \Omega\right\}\!.
\end{equation*}


Finally, by $c$, we denote a generic positive constant.
    
\section{Main assumptions and preliminary results} \label{main_assumptions}

For the rest of this work, let $\Omega\subset \R^3$ be open, bounded, and connected with $\partial\Omega\in \mathscr{C}^{2,\mu}$, $0<\mu<1$. In order to formulate the problem, we first introduce certain function spaces for better readability. Since this work refers to Fern\'andez-Cara, Guill\'en, and Ortega \cite{FCGO02} for the single-valued case, we will mostly use the same notation. Let $1<r,s<\infty$. We define
\begin{equation*}
  \begin{aligned}
    H_r&= \left\{ u\in L^r(\Omega)^3 \mid \nabla\cdot u=0,\ u\cdot n= 0 \text{ on } \partial\Omega\right\}\!,\\
    V_r&= H_r\cap W_0^{1,r}(\Omega)^3=\left\{ u\in W_0^{1,r}(\Omega)^3 \mid \nabla\cdot u=0\right\}\!,
  \end{aligned}
\end{equation*}
where the divergence in the definition of $H_r$ is meant in the distributional sense and where $u\cdot n$ is meant in the sense of traces with $n$ denoting the outer normal unit vector to $\partial\Omega$. Further, by $P_r$, we denote the usual Helmholtz (or Helmholtz-Leray, Helmholtz-Weyl) projector $P_r\colon L^r(\Omega)^3\to H_r$, i.e., $P_r$ is linear and bounded with $P_r u=v$ where $v$ is given by the so-called Helmholtz decomposition
\begin{equation} \label{Helmholtz_decomp}
  u=v + \nabla w
\end{equation}
with $v\in H_r$ and $w\in W^{1,r}(\Omega)$, see, e.g., Galdi~\cite[Chapter III.1]{Galdi}. Based on this, by $A_r$, we denote the Stokes operator $A_r\colon D(A_r)\to H_r$ with the domain $D(A_r)=V_r\cap W^{2,r}(\Omega)^3$ and $A_r u=P_r(-\Delta u)$ for all $u\in D(A_r)$. Equipped with the norm 
\begin{equation*}
  \Vert u \Vert_{D(A_r)}= \Vert u \Vert_{H_r} + \Vert A_r u \Vert_{H_r},
\end{equation*}
$D(A_r)$ is a Banach space, see, e.g., Butzer and Berens~\cite[p.~11]{ButzerBerens}. We also introduce the space
\begin{equation*}
  D_r^s=\left\{ u\in H_r\ \bigg|\ \int_0^\infty \Vert{A_r e^{-tA_r}u}\Vert^s_{H_r}\diff t<\infty\right\}\!,
\end{equation*}
which is, equipped with the norm
\begin{equation*}
  \Vert u \Vert_{D_r^s} = \Vert u \Vert_{H_r} + \left(\int_0^\infty \Vert A_r \,e^{-tA_r}\,u \Vert_{H_r}^s \diff t\right)^{1/s},
\end{equation*}
again a Banach space, coinciding with a real interpolation space between $D(A_r)$ and $H_r$ and with the continuous and dense embeddings
\begin{equation} \label{embeddings_D_r^s}
  D(A_r)\subset D^s_r\subset H_r,
\end{equation}
see, e.g., Butzer and Berens~\cite[Chapter~III]{ButzerBerens}. As mentioned in Fern\'andez-Cara, Guill\'en, and Ortega \cite[p.~563]{FCGO02}, this space is a natural choice for the initial data $u_0$ of our differential inclusion problem if we are looking for a solution in $L^s(0,T;D(A_r))$, see also Giga and Sohr~\cite[pp.~77~f.]{GigaSohr}.

For simplicity, we still write $\partial_t u$, $\nabla u$, or $\nabla\cdot u$ for  abstract functions $u\colon [0,T]\to X$, where $X$ is a Banach space of functions mapping $\Omega$ to $\R$ (or $\R^3$, $\R^{3\times 3}$, respectively). Also, if there is no risk of confusion, we simply write, e.g., $L^r$ and $W^{1,r}$ for the spaces $L^r(\Omega)^3$ and $W^{1,r}(\Omega)^{3\times 3}$ and, e.g., $\Vert\cdot\Vert_{L^s(L^r)}$ for the norm of the space $L^s(0,T;L^r)$.
  
Now, let us state the assumptions on the set-valued right-hand side $F\colon [0,T]\times H_r \to \mathcal{P}_{fc}(L^r)$. We say that the assumptions {\textbf{(F)}} are fulfilled if
\begin{itemize}
  \item[\textbf{(F1)}] $F$ is measurable,
  \item[\textbf{(F2)}] for almost all $t\in(0,T)$, the graph of the mapping $v\mapsto F(t,v)$ is sequentially closed in $H_r\times L^r_w$, where $L^r_w$ denotes the Hilbert space $L^r$ equipped with the weak topology, and
  \item[\textbf{(F3)}] for almost all $t\in(0,T)$ and all $v\in H_r$, we have the estimate $$\vert{F(t,v)}\vert \leq b(t)\left(1 + \gamma\left(\Vert{v}\Vert_{H_r}\right)\right)$$ with $b\in L^s(0,T)$, $b\geq 0$ a.e., and $\gamma\colon[0,\infty)\to [0,\infty)$ a monotonically increasing function.
\end{itemize}

An example for $F$ with $\gamma(z)=c\;\!z^{2/s'}$, $c>0$, can be found in Eikmeier and Emmrich~\cite[Section~5]{EikEmm20}.  
  In order to prove our main result, we need the following
\begin{lemma} \label{MeasurabilityNemytskii}
  Let $X$ and $Y$ be separable Banach spaces. If the set-valued mapping $G\colon [0,T]\times X\to P_f(Y)$ is measurable and the mapping $v\colon [0,T]\to X$ is Bochner measurable, then the set-valued Nemytsky mapping $\tilde{G}_v\colon [0,T]\to P_f(Y)$, $t\mapsto G(t,v(t))$, is measurable.
\end{lemma}

The proof to this lemma with $X=Y$ is given in Eikmeier and Emmrich~\cite[Lemma~2]{EikEmm20}, but it can easily be adapted to the case $X\neq Y$.

\section{Local existence}
  
  We are now able to state our main result about the local existence of strong solutions.
  
\begin{theorem} \label{thm_local}
  Let $\Omega\subset\R^3$ be open, bounded, and connected with $\partial\Omega\in \mathscr{C}^{2,\mu}$, $0<\mu<1$, let $T>0$ and let $u_0\in D^s_r$, $\tau_0\in W^{1,r}$ with $3<r<\infty$, $1<s<\infty$. Let $F\colon [0,T]\times H_r \to \mathcal{P}_{fc}(L^r)$ satisfy the assumptions \textbf{(F)}. Then there exists $T_*>0$ and 
  \begin{equation*}
    \begin{aligned}
      u&\in L^s(0,T_*; D(A_r)) \quad\text{with}\quad \partial_t u\in L^s(0,T_*; H_r),\\
      \tau &\in \mathscr C([0,T_*];W^{1,r})\quad\text{with}\quad \partial_t \tau\in L^s(0,T_*; L^r),\\
      p&\in L^s(0,T_*; W^{1,r}),\\
    \end{aligned}
  \end{equation*}
  such that $(u,\tau,p)$ is a solution to 
  \begin{equation} \label{problem_multivalued}
    \left\{\begin{aligned}
      \mathrm{Re} \left(\partial_t u + (u\cdot\nabla) u\right) -(1-\alpha)\Delta u - \nabla\cdot \tau+ \nabla p &\in F(\cdot ,u) \phantom{=02\alpha D(u)\tau_0f}\hspace{-1.5cm} \text{in } (0,T_*),\\
      \mathrm{We}\left( \partial_t \tau + (u\cdot\nabla)\tau + g_a(\tau,\nabla u)\right) + \tau &= 2\alpha D(u) \phantom{\in 0\tau_0fF(\cdot,u)}\hspace{-1.5cm} \text{in } (0,T_*),\\
      u(0)=u_0, \quad\quad \tau(0)&=\tau_0,
    \end{aligned}\right.
  \end{equation}  
  i.e., a solution to the the single-valued problem 
  \begin{equation} \label{abstract_problem_single-valued}
    \left\{\begin{aligned}
      \mathrm{Re} \left(\partial_t u + (u\cdot\nabla) u\right) -(1-\alpha)\Delta u -\nabla\cdot \tau+ \nabla p &= f \phantom{02\alpha D(u)\tau_0}\hspace{-.3cm} \text{in } (0,T_*),\\
      \mathrm{We}\left( \partial_t \tau + (u\cdot\nabla)\tau + g_a(\tau,\nabla u)\right) + \tau &= 2\alpha D(u) \phantom{0\tau_0f}\hspace{-.3cm} \text{in } (0,T_*),\\
      u(0)=u_0, \quad\quad \tau(0)&=\tau_0,
  \end{aligned}\right.
  \end{equation}
  where $f\in L^s(0,T_*; L^r)$ with $f(t)\in F(t,u(t))$ a.e. in $(0,T_*)$.
\end{theorem}

\begin{proof}
  For $\T>0$, we introduce the spaces
  \begin{equation*}
    \begin{aligned}
      \mathcal{U}(\T)&:= \left\{ u\in L^s(0,\T;D(A_r))\mid \partial_t u\in L^s(0,\T;H_r)\right\},\\
      \mathcal{T}(\T)&:= \left\{ \tau\in L^\infty(0,\T;W^{1,r})\mid \partial_t \tau\in L^s(0,\T;L^r)\right\},
    \end{aligned}
  \end{equation*}
  and
  \begin{equation*}
    \mathcal{W}(\T):= \mathcal{U}(\T) \times \mathcal{T}(\T)
  \end{equation*}
  with the norm
  \begin{equation*}
    \Vert{(u,\tau)}\Vert_{\mathcal{W}(\T)}:= \Vert{u}\Vert_{L^s(D(A_r))}+\Vert{\partial_t u}\Vert_{L^s(H_r)} + \Vert{\tau}\Vert_{L^\infty(W^{1,r})} + \Vert{\partial_t \tau}\Vert_{L^s(L^r)}.
  \end{equation*}
  For $R_i>0$, $i=1,2,3$, let 
  \begin{equation} \label{Y(T)}
    \begin{aligned}
      \mathcal{Y}(\T):= \{ (u,\tau)\in \mathcal{W}(\T) \mid\ &u(0)=u_0, \quad \tau(0)=\tau_0,\\
      & \!\Vert{u}\Vert_{L^s(D(A_r))}^s+ \Vert{\partial_t u}\Vert^s_{L^s(H_r)}\leq R_1^s,\\
      &\!\left.\! \Vert{\tau}\Vert_{L^\infty(W^{1,r})} \leq R_2, \quad \Vert{\partial_t \tau}\Vert_{L^s(L^r)} \leq R_3 \right\}.
    \end{aligned}
  \end{equation}
  There exists $c_1>0$, depending on Re, $r$, $s$, and $\Omega$, such that for
  \begin{equation} \label{estimate_R_1_R_2}
    R_1\geq \frac{c_1}{1-\alpha} \Vert u_0\Vert_{D^s_r}, \quad R_2\geq \Vert \tau_0\Vert_{W^{1,r}},
  \end{equation}
  the set $\mathcal{Y}(\T)$ is nonempty for all $\T>0$, see the proof for the single-valued case in Fern\'andez-Cara, Guill\'en, and Ortega~\cite[p.~570]{FCGO02}. Let also
  \begin{equation*}
    \mathcal{X}(\T):= L^s(0,\T; V_r)\times \mathscr{C}([0,\T];L^r).
  \end{equation*}
  Now, for $0<\T\leq T$, let $$\Phi\colon \mathcal{Y}(\T)\to \mathcal{P}_c(\mathcal{X}(\T))$$ with $$(u,\tau)\in \Phi(\tilde{u}, \tilde{\tau}), \quad (\tilde{u}, \tilde{\tau})\in \mathcal{Y}(\T),$$ iff $u\in \mathcal{U}(\T)$ is a solution to
  \begin{equation}\label{linearised_equation_u}
    \left\{\begin{aligned}
      \text{Re}\;\! \partial_t u +(1-\alpha)\;\!A_r u &\in P_r\left(-\text{Re}\;\! (\tilde{u}\cdot\nabla)\tilde{u} + \nabla\cdot \tilde{\tau} + F(\cdot,\tilde{u})\right) \phantom{=0u_0}\hspace{-.5cm}\text{in } (0,\T),\\
      u(0)&=u_0,
    \end{aligned}\right.
  \end{equation}        
  and $\tau\in \mathcal{T}(\T)$ is a solution to 
  \begin{equation}\label{linearised_equation_tau}
    \left\{\begin{aligned}        
      \text{We}\left( \partial_t \tau + (\tilde{u}\cdot\nabla)\tau + g_a(\tau,\nabla\tilde{u})\right) + \tau &= 2\alpha D(\tilde{u}) \phantom{\tau_0} \text{in } (0,\T) ,\\
      \tau(0)&=\tau_0 .
    \end{aligned}\right.
  \end{equation}
  Since $\tilde{u}$ and $\tilde{\tau}$ are fixed, the two systems~\eqref{linearised_equation_u} and~\eqref{linearised_equation_tau} are linear in $u$ and $\tau$, respectively. First, we show that a fixed point of $\Phi$, i.e., a pair $(u,\tau)\in \mathcal{Y}(\T)$ with $(u,\tau)\in \Phi(u,\tau)$, implies the existence of a solution $(u,\tau,p)$ to \eqref{problem_multivalued}: Let $(u,\tau)\in \mathcal{Y}(\T)$ be a fixed point of $\Phi$. Then, there exists $f\in \mathcal{F}^s(u)$ such that 
  \begin{equation*}
    \text{Re}\;\! \partial_t u +(1-\alpha)\;\!A_r u = P_r\left(-\text{Re}\;\! (u\cdot\nabla)u + \nabla\cdot \tau + f\right) \phantom{=0u_0}\hspace{-.5cm}\text{in } (0,\T).
  \end{equation*}
  Since $\partial_t u\in L^s(0,T;H_r)$, we have $P_r(\partial_t u)= \partial_t u$. Together with the definition of the Stokes operator $A_r$, we obtain
  \begin{equation*}
    P_r\left(\text{Re}\;\! \partial_t u - (1-\alpha) \Delta u + \text{Re}\;\! (u\cdot\nabla)u - \nabla\cdot \tau - f\right)=0 \phantom{=0u_0}\hspace{-.5cm}\text{in } (0,\T).
  \end{equation*}
  Now, the Helmholtz decomposition, see~\eqref{Helmholtz_decomp}, implies that, for almost all $t\in (0,\T)$, there exists $p(t)\in W^{1,r}$ such that 
  \begin{equation*}
    \text{Re}\;\! \partial_t u - (1-\alpha) \Delta u + \text{Re}\;\! (u\cdot\nabla)u - \nabla\cdot \tau - f = \nabla p\phantom{=0u_0}\hspace{-.5cm}\text{in } (0,\T).
  \end{equation*}
  Due to the regularity of $u$, $\tau$, and $f$, we obtain $p\in L^s(0,\T;W^{1,r})$. Finally, it is easy to see that $\tau$ solves the second equation in~\eqref{problem_multivalued}, so $(u,\tau,p)$ is a solution to~\eqref{problem_multivalued}.
  
  In order to show the existence of a fixed point, we want to apply the generalisation of Kakutani's fixed-point theorem (see Glicksberg~\cite{Glicksberg} and Fan~\cite{Fan}), i.e., we have to show that there exists $T_*$ with $0<T_*\leq T$ such that $\mathcal{Y}(T_*)$ is nonempty, convex, and compact in $\mathcal{X}(T_*)$, that $\Phi$ maps $\mathcal{Y}(T_*)$ into convex subsets of $\mathcal{Y}(T_*)$, and that $\Phi$ is closed, i.e., its graph is closed in $\mathcal{X}(T_*)\times \mathcal{X}(T_*)$.
  
  First, we show that $\Phi$ is well-defined, i.e., that for all $(\tilde{u},\tilde{\tau})\in \mathcal{Y}(\T)$, the set $\Phi(\tilde{u},\tilde{\tau})$ is nonempty and convex in $\mathcal{X}(\T)$. Due to $\tilde{u}\in \mathcal{U}(\T)$, we have $\tilde{u}\in \mathscr{AC}([0,\T];H_r)$, i.e., $\tilde{u}$ is Bochner measurable and the mapping $t\mapsto \Vert \tilde{u}(t)\Vert_{H_r}$ is bounded. Following Lemma \ref{MeasurabilityNemytskii}, the mapping $t\mapsto F(t, \tilde{u}(t))$ is measurable and there exists a measurable selection $f$. Due to Assumption \textbf{(F3)}, we have
  \begin{equation*}
    \Vert f(t)\Vert_{L^r}\leq b(t)\left( 1 + \gamma\left(\Vert \tilde{u}(t)\Vert_{H_r}\right)\right)\leq b(t)\left( 1+\gamma(c)\right)
  \end{equation*}    
  for some $c>0$ and for almost all $t\in (0,\T)$, implying $f\in L^s(0,\T; L^r)$ (due to $b\in L^s(0,T)$) and thus $f\in \mathcal{F}^s(\tilde{u})$. Now, we can apply Fern\'andez-Cara, Guill\'en, and Ortega~\cite[Lemma~10.1]{FCGO02} to obtain a (unique) solution $u\in \mathcal{U}(\T)$ to the single-valued problem
  \begin{equation} \label{linearised_equation_u_short}
    \left\{\begin{aligned}
      \text{Re}\;\! \partial_t u +(1-\alpha)\;\!A_r u &=P_r\left(-\text{Re}\;\! (\tilde{u}\cdot\nabla)\tilde{u} + \nabla\cdot \tilde{\tau} + f\right) \phantom{=0u_0}\hspace{-.5cm}\text{in } (0,\T),\\
      u(0)&=u_0.
    \end{aligned}\right.
  \end{equation}
  Also, following Fern\'andez-Cara, Guill\'en, and Ortega~\cite[Lemma~10.3]{FCGO02}, there exists a (unique) solution $\tau\in \mathcal{T}(\T)$ to \eqref{linearised_equation_tau}. Therefore, $\Phi(\tilde{u},\tilde{\tau})$ is nonempty. In order to show the convexity of $\Phi(\tilde{u},\tilde{\tau})$, let $(\tilde{u}, \tilde{\tau})\in \mathcal{Y}(\T)$, $\lambda\in(0,1)$, and $(u_1, \tau_1),(u_2,\tau_2)\in \Phi(\tilde{u}, \tilde{\tau})$. Therefore, there exist $f_1,f_2\in \mathcal{F}^s(\tilde{u})$ such that $u_i$ is a solution to the single-valued problem~\eqref{linearised_equation_u_short} with $f_i$ instead of $f$ on the right-hand side, $i=1,2$. Now, the linearity of problem~\eqref{linearised_equation_u_short} in $u$ (remember that $\tilde{u}$ is fixed) implies that $\lambda u_1+ (1-\lambda)u_2$ is a solution to problem~\eqref{linearised_equation_u_short} with $\lambda f_1 +(1-\lambda)f_2$ on the right-hand side. Since $F$ is convex-valued, it is easy to show that the set $\mathcal{F}^s(\tilde{u})$ is convex as well. Thus, we have $\lambda f_1 +(1-\lambda)f_2\in \mathcal{F}^s(\tilde{u})$. As a consequence, $\lambda u_1+ (1-\lambda)u_2$ is a solution to problem~\eqref{linearised_equation_u}. Also, due to the linearity of problem~\eqref{linearised_equation_tau} in $\tau$, $\lambda\tau_1+(1-\lambda)\tau_2$ is a solution to this problem. Overall, this implies $\lambda (u_1,\tau_1)+(1-\lambda)(u_2,\tau_2)\in \Phi(\tilde{u}, \tilde{\tau})$, i.e., $\Phi$ is convex-valued. Therefore, $\Phi$ is well-defined.
   
  For all $\T>0$, the set $\mathcal{Y}(\T)$ is a nonempty, convex, and compact subset of $\mathcal{X}(\T)$, see the proof for the single-valued case in Fern\'andez-Cara, Guill\'en, and Ortega~\cite[Theorem~9.1]{FCGO02}. Next, we have to show that there exist $T_*,R_1, R_2,R_3>0$ such that
  \begin{equation*}
    \Phi\colon \mathcal{Y}(T_*)\to \mathcal{P}_c(\mathcal{Y}(T_*)).
  \end{equation*}
  We approach similarly to the proof for the single-valued case, but we have to include the estimates on our set-valued operator $F$. Let $(u,\tau)\in \Phi(\tilde{u}, \tilde{\tau})$ for an arbitrary $(\tilde{u}, \tilde{\tau})\in \mathcal{Y}(\T)$ and let $f\in \mathcal{F}^s(\tilde{u})$ such that $u$ solves the single-valued problem~\eqref{linearised_equation_u_short} with $f$. Following Fern\'andez-Cara, Guill\'en, and Ortega~\cite[p.~571]{FCGO02}, an a priori estimate for the solution to the single-valued problem~\eqref{linearised_equation_u_short} (cf. Fern\'andez-Cara, Guill\'en, and Ortega~\cite[Lemma~10.1]{FCGO02}) yields
  \begin{equation} \label{a_priori_u_prelim}
    \begin{aligned}
      \Vert u \Vert_{L^s(D(A_r))}^s + \Vert \partial_t u \Vert_{L^s(H_r)}^s &\leq \left(\frac{c_1}{1-\alpha}\right)^s\left( \Vert u_0 \Vert_{D^s_r}^s + c_2 \left( \Vert u_0 \Vert_{H_r}^{3s(r-1)/(2r)} R_1^{s(r+3)/(2r)}\T^{(r-3)/(2r)}   \right.\right. \\
      &\hspace{3cm} \left.\left. + R_1^{2s} \T^{3s(r-1)/(2r)-1} + R_2^s\T +\Vert f \Vert_{L^s(L^r)}^s \right)\right)
    \end{aligned}
  \end{equation}
  with $c_2>0$ depending on Re, $r$, $s$, and $\Omega$. Using Assumption \textbf{(F3)}, we obtain
  \begin{equation} \label{estimate_f}
    \begin{aligned}
      \Vert f \Vert_{L^s(L^r)}^s &= \int_0^{\T} \Vert f(t)\Vert_{L^r}^s\diff t\\
      &\leq \int_0^{\T} b(t)^s\left( 1 + \gamma\!\left(\Vert \tilde{u}(t) \Vert_{H_r}\right) \right)^s \diff t.\\
    \end{aligned}
  \end{equation}
  As mentioned before, we have $u\in \mathscr{AC}(0,\T;H_r)$ and in particular $u\in L^\infty(0,\T;H_r)$. Hölder's inequality yields
  \begin{equation} \label{estimate_L_infty_H_r}
    \Vert \tilde{u} \Vert_{L^\infty(H_r)} \leq \Vert u_0 \Vert_{H_r} + \Vert \partial_t \tilde{u} \Vert_{L^1(H_r)}\leq \Vert u_0 \Vert_{H_r} + \T^{1/s'} \Vert \partial_t \tilde{u} \Vert_{L^s(H_r)}.
  \end{equation}
  Together with \eqref{estimate_f} and with the monotonicity of $\gamma$, we have
  \begin{equation*}
    \begin{aligned}
      \Vert f \Vert_{L^s(L^r)}^s &\leq \Vert b \Vert_{L^s(0,T)}^s \left( 1+\gamma\! \left( \Vert u_0 \Vert_{H_r} + \T^{1/s'} \Vert \partial_t \tilde{u} \Vert_{L^s(H_r)} \right)\right)^s ,
    \end{aligned}
  \end{equation*}
  so $\tilde{u}\in \mathcal{Y}(\T)$ implies
  \begin{equation} \label{estimate_f_2}
    \begin{aligned}
      \Vert f \Vert_{L^s(L^r)}^s &\leq \Vert b \Vert_{L^s(0,T)}^s \left( 1 + \gamma\! \left( \Vert u_0 \Vert_{H_r} + R_1 \T^{1/s'} \right)\right)^s.
    \end{aligned}
  \end{equation}
  Combined with \eqref{a_priori_u_prelim}, we end up with
  \begin{equation} \label{a_priori_u}
    \begin{aligned}
      \Vert u \Vert_{L^s(D(A_r))}^s + \Vert \partial_t u \Vert_{L^s(H_r)}^s &\leq \left(\frac{c_1}{1-\alpha}\right)^s\left( \Vert u_0 \Vert_{D^s_r}^s + c_2 \left( \Vert u_0 \Vert_{H_r}^{3s(r-1)/(2r)} R_1^{s(r+3)/(2r)}\T^{(r-3)/(2r)}   \right.\right. \\
      &\hspace{2cm}  + R_1^{2s} \T^{3s(r-1)/(2r)-1} + R_2^s\T \\
      &\hspace{2.7cm} \left.\left.+  \Vert b \Vert_{L^s(0,T)}^s \left( 1 + \gamma\! \left( \Vert u_0 \Vert_{H_r} + R_1 \T^{1/s'} \right)\right)^s \right)\right).
    \end{aligned}
  \end{equation}
  The estimates for $\tau$ are obtained by the a priori estimates for the solution to problem~\eqref{linearised_equation_tau} (cf. Fern\'andez-Cara, Guill\'en, and Ortega~\cite[Lemma~10.3]{FCGO02}), so we have
  \begin{equation} \label{a_priori_tau}
    \Vert \tau \Vert_{L^\infty(W^{1,r})}\leq \left( \Vert \tau_0 \Vert_{W^{1,r}} + \frac{4\alpha}{c_3\mathrm{We}}\right)\exp\left(c_3\;\!R_1\T^{1/s'}\right) =: \Lambda
  \end{equation}
  and 
  \begin{equation} \label{a_priori_dt_tau}
    \Vert \partial_t \tau \Vert_{L^s(L^r)}\leq c_4 \;\!\Lambda \left(R_1+ \frac{\T^{1/s}}{c_3\mathrm{We}}\right)
  \end{equation}
  with $c_3,c_4>0$ depending on $a$, $r$, $s$, and $\Omega$. Due to the last three estimates, we can choose $T_*$ small enough and $R_1$, $R_2$, and $R_3$ large enough, respectively, such that $(\tilde{u},\tilde{\tau})\in \mathcal{Y}(T_*)$ implies $(u,\tau)\in \mathcal{Y}(T_*)$ for all $(u,\tau)\in \Phi(\tilde{u},\tilde{\tau})$, i.e., 
  \begin{equation*}
    \Phi(\mathcal{Y}(T_*))\subset \mathcal{Y}(T_*).
  \end{equation*}
  A possible choice for $T_*$, $R_1$, $R_2$, and $R_3$ is, e.g.,
  \begin{equation*}
    \begin{aligned}
      R_1^s &= \left(\frac{c_1}{1-\alpha}\right)^s \left(\Vert u_0 \Vert_{D^s_r}^s + c_2 \left( \Vert u_0 \Vert_{H_r}^{3s(r-1)/(2r)} +2 +\Vert b \Vert_{L^s(0,T)}^s \left( 1+ \gamma\! \left( \Vert u_0 \Vert_{H_r} +1\right)\right)^s\right)\right),\\
      R_2 &= \left( \Vert \tau_0 \Vert_{W^{1,r}} + \frac{4\alpha}{c_3\mathrm{We}}\right) \exp c_3, \\
      R_3 &= c_4\;\!R_2 \left(R_1+ \frac{1}{c_3\mathrm{We}\;\!R_2}\right),\\
      T_* &= \min\left( R_1^{-s(r+3)/(r-3)},\ R_1^{-4sr/(3s(r-1)-2r)},\ R_2^{-s},\ R_1^{-s'},\ T \right).
    \end{aligned}
  \end{equation*}
  In particular, $R_1$ and $R_2$ fulfil the necessary estimates~\eqref{estimate_R_1_R_2}.
    
  Finally, we have to show that $\Phi$ is closed, i.e., its graph is closed in $\mathcal{X}(T_*)\times \mathcal{X}(T_*)$. Let $(\tilde{u}_n, \tilde{\tau}_n,u_n,\tau_n)\subset \text{graph}\;\! \Phi$ with $(\tilde{u}_n, \tilde{\tau}_n,u_n,\tau_n)\to (\tilde{u},\tilde{\tau}, u,\tau)$ in $\mathcal{X}(T_*)\times \mathcal{X}(T_*)$, so in particular 
  \begin{equation} \label{strong_convergences}
    \begin{aligned}
      \tilde{u}_n\to \tilde{u}, ~~ u_n\to u \quad & \text{in } L^s(0,T_*;V_r),\\
      \tilde{\tau}_n\to \tilde{\tau}, ~~\tau_n\to\tau \quad &\text{in } \mathscr{C}([0,T_*];L^r).
    \end{aligned}
  \end{equation}
  
   We have to show $(\tilde{u},\tilde{\tau}, u,\tau)\in \text{graph}\;\!\Phi$, i.e., $(u,\tau)\in \Phi(\tilde{u}, \tilde{\tau})$. Due to $\Phi(\mathcal{Y}(T_*))\subset \mathcal{Y}(T_*)$, we know $(\tilde{u}_n, \tilde{\tau}_n,u_n,\tau_n)\in \mathcal{Y}(T_*)\times \mathcal{Y}(T_*)$. Since the spaces $L^s(0,T_*;D(A_r))$, $L^s(0,T_*;H_r)$, and $L^s(0,T_*;L^r)$ are reflexive Banach spaces and the space $L^\infty(0,T_*;W^{1,r})$ is the dual space of a separable normed space, the boundedness of $\mathcal{Y}(T_*)$ in $\mathcal{W}(T_*)$ implies that there exist subsequences (again denoted by $n$) and $v\in L^s(0,T_*;D(A_r))$, $w\in L^s(0,T_*;H_r)$, $\eta\in L^\infty(0,T_*;W^{1,r})$, and $\theta\in L^s(0,T_*;L^r)$ such that
  \begin{equation} \label{weak_convergences}
    \begin{aligned}
      u_n \wconv v\quad &\text{in } L^s(0,T_*;D(A_r)),\\
      \partial_t u_n\wconv w\quad & \text{in } L^s(0,T_*;H_r), \\
      \tau_n \wsconv \eta \quad &\text{in }  L^\infty(0,T_*;W^{1,r}), \\
      \partial_t \tau_n \wconv \theta \quad &\text{in } L^s(0,T_*;L^r),
    	\end{aligned}
  \end{equation}
  and the same for $\tilde{u}_n$ (with $\tilde{v}\in L^s(0,T_*;D(A_r))$) etc. It is easy to see that $w=\partial_t v$ and $\theta=\partial_t \eta$. Due to the uniqueness of the weak limit, it is also easy to see that $\tilde{v}=\tilde{u}$, $\tilde{\eta}=\tilde{\tau}$, $v=u$, and $\eta=\tau$.
  
  Since $(\tilde{u}_n, \tilde{\tau}_n,u_n,\tau_n)\in \text{graph}\;\! \Phi$, $n\in \N$, there exists $f_n\in \mathcal{F}^s(\tilde{u}_n)$ such that
  \begin{equation} \label{problem_n}
    \left\{\begin{aligned}
      \text{Re}\;\! \partial_t u_n + (1-\alpha)\;\!A_r u_n &= P_r(-\text{Re}\;\! (\tilde{u}_n\cdot\nabla)\tilde{u}_n + \nabla\cdot \tilde{\tau}_n + f_n) \phantom{2 \alpha D(\tilde{u}_n)}\hspace{-.7cm}\text{in } (0,T_*),\\
      u_n(0)&=u_0, \\
      \text{We}\left( \partial_t \tau_n + (\tilde{u}_n\cdot\nabla)\tau_n + g_a(\tau_n,\nabla\tilde{u}_n)\right) + \tau_n &= 2\alpha D(\tilde{u}_n) \phantom{ P_r(-\text{Re}\;\! (\tilde{u}_n\cdot\nabla)\tilde{u}_n + \nabla\cdot \tilde{\tau}_n + f_n) } \hspace{-.7cm} \text{in } (0,T_*) ,\\
      \tau_n(0)&=\tau_0 .
    \end{aligned}\right.
  \end{equation}
  As derived in \eqref{estimate_f_2}, $f_n\in \mathcal{F}^s(\tilde{u}_n)$ and Assumption \textbf{(F3)} yield
  \begin{equation*}
    \Vert f_n \Vert_{L^s(L^r)}^s \leq \Vert b \Vert_{L^s(0,T_*)}^s \left( 1+ \gamma\! \left( \Vert u_0 \Vert_{H_r} + R_1 T_*^{1/s'} \right)\right)^s,
  \end{equation*}
  so the sequence $(f_n)\subset L^s(0,T_*;L^r)$ is bounded. Again, due to the reflexivity of $L^s(0,T_*;L^r)$, there exist a subsequence of the subsequence (again denoted by $n$) and $f\in L^s(0,T_*;L^r)$ such that 
  \begin{equation*}
    f_n\wconv f \quad \text{in } L^s(0,T_*;L^r).
  \end{equation*}
  Similarly to \eqref{estimate_f_2}, we can also derive the pointwise estimate
  \begin{equation*}
    \Vert f_n(t)\Vert_{L^r} \leq b(t) \left( 1+ \gamma\!\left(\Vert u_0\Vert_{H_r}+ R_1 T_*^{1/s'}\right)\right) =:\tilde{b}(t)
  \end{equation*}  
  for almost all $t\in(0,T_*)$. Therefore, we have $f_n(t)\in \bar{B}_{L^r}(0;\tilde{b}(t))$ for almost all $t\in(0,T_*)$, where $\bar{B}_{L^r}(0;\tilde{b}(t))$ denotes the closed ball in $L^r$ around $0$ of radius $\tilde{b}(t)$. This implies
  \begin{equation*}
    f(t)\in \overline{\text{co}}\left( \wlim\{f_n(t)\}\right)
  \end{equation*}
  for almost all $t\in (0,T_*)$, where $\wlim$ denotes the weak Kuratowski limit superior of a sequence of sets, i.e., for a sequence $(M_n)\subset 2^X$ of subsets of a Banach space $X$, we have
  \begin{equation*}
    \wlim M_n = \{x\in X\mid \exists (x_k) \subset X,~ x_k\in M_{n_k},~ k\in\N,~ x_k\wconv x \text{ in } X\},
  \end{equation*}  
  see, e.g., Papageorgiou~\cite[Theorem~3.1]{Papageorgiou87}. As we have $f_n(t)\in F(t,\tilde{u}_n(t))$ for almost all $t\in (0,T_*)$ and all $n\in\N$, we obtain
  \begin{equation} \label{f_in_co}
    f(t)\in \overline{\text{co}}\left( \wlim F(t,\tilde{u}_n(t))\right)
  \end{equation}  
  for almost all $t\in (0,T_*)$. Since $(\tilde{u}_n,\tilde{\tau}_n)\to (\tilde{u}, \tilde{\tau})$ in $\mathcal{X}(T_*)$ and in particular $\tilde{u}_n\to \tilde{u}$ in $L^s(0,T_*;V_r)$, we have, up to a subsequence, $\tilde{u}_n(t)\to \tilde{u}(t)$ in $V_r$ for almost all $t\in (0,T_*)$, see, e.g., Brezis~\cite[Theorem~4.9]{Brezis}. Now, Assumption~\textbf{(F2)} implies 
  \begin{equation*}
    \wlim F(t,\tilde{u}_n(t)) \subset F(t,\tilde{u}(t))
  \end{equation*}
  for almost all $t\in (0,T_*)$. As $F(t,\tilde{u}(t))$ is closed and convex, this yields, together with~\eqref{f_in_co},   
  \begin{equation} \label{f_in_F}
    f(t)\in F(t,\tilde{u}(t))
  \end{equation}
  for almost all $t\in(0,T_*)$.
  
  It remains to show that we can pass to the limit in \eqref{problem_n}, i.e., that $u$ and $\tau$ solve \eqref{linearised_equation_u_short} and \eqref{linearised_equation_tau}, respectively. Due to the weak convergences established in \eqref{weak_convergences}, we immediately obtain the convergences
  \begin{equation*}
    \begin{aligned}
      \partial_t u_n\wconv \partial_t u\quad & \text{in } L^s(0,T_*;H_r), \\
      \nabla\cdot \tilde{\tau}_n\wsconv \nabla\cdot\tilde{\tau} \quad & \text{in } L^\infty(0,T_*;L^r),\\
      \partial_t \tau_n \wconv \partial_t \tau \quad &\text{in } L^s(0,T_*;L^r),\\
      \tau_n \wsconv \tau \quad &\text{in }  L^\infty(0,T_*;W^{1,r}), \\
      D(\tilde{u}_n)\wconv D(\tilde{u}) \quad &\text{in } L^s(0,T_*;V_r).
    	\end{aligned}
  \end{equation*}
  We also have $A_r u_n\wconv A_r u$ in $L^s(0,T_*;H_r)$ since $u_n\wconv u$ in $L^s(0,T_*;D(A_r))$ and since $A_r\colon D(A_r)\to H_r$ is a linear, bounded operator and therefore weakly sequentially continuous, see, e.g., Zeidler~\cite[Proposition~21.81]{ZeidlerIIA}. In particular, all these convergences imply the weak convergence of the respective terms in $L^s(0,T_*;L^{r/2})$. Since we only want to show that $u$ and $\tau$ are a solution to~\eqref{linearised_equation_u} and~\eqref{linearised_equation_tau}, i.e., that the respective equations are fulfilled almost everywhere, it suffices to pass to the limit in~\eqref{problem_n} in $L^s(0,T_*;L^{r/2})$ (it would even be enough to pass to the limit in $L^1(0,T_*; L^1)$).
  
  First, we show the weak convergence $(\tilde{u}_n\cdot\nabla)\tilde{u}_n \wconv (\tilde{u}\cdot\nabla)\tilde{u}$ in $L^s(0,T_*;L^{r/2})$. Let $\varphi\in L^{s'}(0,T_*;L^{r/(r-2)})$. We have
  \begin{equation*}
    \begin{aligned}
      \langle (\tilde{u}_n\cdot\nabla)\tilde{u}_n -(\tilde{u}\cdot\nabla)\tilde{u},\varphi\rangle &= \langle (\tilde{u}_n\cdot\nabla)\tilde{u}_n -(\tilde{u}_n\cdot\nabla)\tilde{u},\varphi\rangle + \langle (\tilde{u}_n\cdot\nabla)\tilde{u}-(\tilde{u}\cdot\nabla)\tilde{u},\varphi\rangle \\
      &\leq \Vert \tilde{u}_n\Vert_{L^{\infty}(H_r)} \,\Vert \tilde{u}_n-\tilde{u}\Vert_{L^s(V_r)} \,\Vert \varphi\Vert_{L^{s'}(L^{r/(r-2)})} \\
      &\hspace{2.5cm} + \int_0^{T_*}\int_\Omega (\tilde{u}_n-\tilde{u})\cdot (\nabla \tilde{u}\, \varphi)\diff x\diff t. \\
    \end{aligned}
  \end{equation*}
  Due to the embeddings $\mathcal{U}(T_*)\subset W^{1,s}(0,T_*;H_r)\subset L^\infty(0,T_*;H_r)$, the norm $\Vert \tilde{u}_n\Vert_{L^{\infty}(H_r)}$ is bounded, so the (strong) convergence $\tilde{u}_n\to \tilde{u}$ in $L^s(0,T_*;V_r)$ (cf. \eqref{strong_convergences}) implies that the first term vanishes as $n\to\infty$. The second term vanishes as well since $\nabla \tilde{u} \,\varphi\in L^1(0,T_*;L^{r/(r-1)})$ and $\tilde{u}_n\wconv \tilde{u}$ in $\mathcal{U}(T_*)$ and thus $\tilde{u}_n\wsconv \tilde{u}$ in $L^\infty(0,T_*;H_r)$.

  Similarly, we show the weak convergence of $(\tilde{u}_n\cdot\nabla)\tau_n$ to $(\tilde{u}\cdot\nabla)\tau$ in $L^s(0,T_*;L^{r/2})$. Let $\varphi\in L^{s'}(0,T_*;L^{r/(r-2)})$. We obtain
  \begin{equation*}
    \begin{aligned}
      \langle (\tilde{u}_n\cdot\nabla)\tau_n -(\tilde{u}\cdot\nabla)\tau,\varphi\rangle &= \langle (\tilde{u}_n\cdot\nabla)\tau_n -(\tilde{u}\cdot\nabla)\tau_n,\varphi\rangle + \langle (\tilde{u}\cdot\nabla)\tau_n-(\tilde{u}\cdot\nabla)\tau,\varphi\rangle \\
      &\leq \Vert \tilde{u}_n-\tilde{u}\Vert_{L^{s}(H_r)} \,\Vert \tau_n\Vert_{L^\infty(W^{1,r})} \,\Vert \varphi\Vert_{L^{s'}(L^{r/(r-2)})} \\
      &\hspace{2.5cm}+ \int_0^{T_*}\int_\Omega \tilde{u}\, \nabla( \tau_n-\tau)\, \varphi\diff x\diff t.
    \end{aligned}
  \end{equation*}
  Again, due to the strong convergence $\tilde{u}_n\to \tilde{u}$ in $L^s(0,T_*;V_r)$ and the weak convergence and thus the boundedness of $(\tau_n)$ in $L^\infty(0,T_*;W^{1,r})$, the first term vanishes as $n\to\infty$. Also, $\tilde{u}\,\varphi\in L^1(0,T_*;L^{r/(r-1)})$ and $\tau_n\wsconv \tau$ in $L^\infty(0,T_*;W^{1,r})$ imply that the second term vanishes as well as $n\to\infty$.
  
  Next, we show the weak convergence $g_a(\tau_n,\nabla\tilde{u}_n)\wconv g_a(\tau,\nabla\tilde{u})$ in $L^s(0,T_*;L^{r/2})$. Since $g_a(\tau_n,\nabla\tilde{u}_n)$ is a linear combination of $\tau_n\, \nabla \tilde{u}_n$ and $\nabla \tilde{u}_n\, \tau_n$, it is sufficient to show the weak convergences of these two terms. Let $\varphi\in L^{s'}(0,T_*;L^{r/(r-2)})$. We have
  \begin{equation*}
    \begin{aligned}
      \langle\tau_n\, \nabla \tilde{u}_n-\tau\,\nabla\tilde{u}, \varphi\rangle & = \langle \tau_n\, (\nabla \tilde{u}_n - \nabla \tilde{u}),\varphi\rangle + \langle (\tau_n- \tau)\,\nabla\tilde{u},\varphi\rangle \\
      & \leq \Vert \tau_n\Vert_{L^\infty(L^r)} \,\Vert \tilde{u}_n-\tilde{u}\Vert_{L^s(V_r)} \, \Vert\varphi\Vert_{L^{s'}(L^{r/(r-2)})} + \int_0^{T_*} \int_\Omega (\tau_n-\tau) \, \nabla\tilde{u}\, \varphi \diff x \diff t.
    \end{aligned}
  \end{equation*}
  Similar to before, the weak convergence $\tau_n\wsconv \tau$ in $L^\infty(0,T_*;W^{1,r})$ implies the boundedness of the norm $\Vert \tau_n\Vert_{L^\infty(L^r)}$. Then, the strong convergence $\tilde{u}_n\to \tilde{u}$ in $L^s(0,T_*;V_r)$ implies that the first term vanishes as $n\to \infty$. The second term vanishes since $\nabla \tilde{u} \,\varphi\in L^1(0,T_*;L^{r/(r-1)})$ and, again, $\tau_n\wsconv \tau $ in $L^\infty(0,T_*;W^{1,r})$. Thus, $\tau_n\, \nabla \tilde{u}_n\wconv\tau\,\nabla\tilde{u}$ in $L^s(0,T_*;L^{r/2})$. Obviously, the weak convergence $\nabla\tilde{u}_n \,\tau_n \wconv \nabla\tilde{u}\,\tau$ in $L^s(0,T_*;L^{r/2})$ can be proven analogously.
  
  Finally, we have to show the convergences of the initial values. The strong convergence $\tau_n\to\tau $ in $\mathscr{C}([0,T_*];L^r)$ (cf.~\eqref{strong_convergences}) immediately implies $\tau_n(0)\to \tau(0)$ in $L^r$. To prove the convergence of $(u_n(0))$, let $\varphi\in \mathscr{C}^1([0,T_*];H_r^*)$. Then, the weak convergences proven before and integration by parts yield
  \begin{equation*}
    \begin{aligned}
      &\langle u_n(T_*), \varphi(T_*)\rangle - \langle u_n(0),\varphi(0)\rangle \\
      & = \int_0^{T_*} \langle \partial_t u_n(t),\varphi(t)\rangle\diff t + \int_0^{T_*} \langle u_n(t) , \partial_t \varphi(t)\rangle\diff t \\
      & \rightarrow \int_0^{T_*} \langle \partial_t u(t),\varphi(t)\rangle\diff t + \int_0^{T_*} \langle u(t) , \partial_t \varphi(t)\rangle\diff t \\
      & = \langle u(T_*), \varphi(T_*)\rangle - \langle u(0),\varphi(0)\rangle
    \end{aligned}
  \end{equation*}
  as $n\to \infty$. Choosing $\varphi(t)=(1-\frac{t}{T_*})\, \sigma$ for an arbitrary $\sigma \in H_r^*$ implies 
  \begin{equation} \label{convergence_initial_value}
    \langle u_n(0),\sigma\rangle \to \langle u(0),\sigma\rangle
  \end{equation}
  for all $\sigma\in H_r^*$. Since we have $u_n(0)=u_0$ for all $n\in \N$, we obtain $u(0)=u_0$ in $H_r$.

  Overall, we have shown that $u$ and $\tau$ are a solution to \eqref{linearised_equation_u_short} and \eqref{linearised_equation_tau}, respectively, so together with \eqref{f_in_F}, we have $(u,\tau)\in \Phi(\tilde{u}, \tilde{\tau})$.\qed
\end{proof}

\section{Global existence for small data}

As in the single-valued case (cf. Fern\'andez-Cara, Guill\'en, and Ortega~\cite[Theorem~9.2]{FCGO02}), global existence of strong solutions can be obtained for small data. Since we cannot control $\Vert \tilde{u} \Vert_{L^\infty(H_r)}$ by choosing $T_*$ small enough anymore (cf.~estimate~\eqref{estimate_L_infty_H_r}), this requires a more specific growth condition on the set-valued right-hand side $F\colon [0,T]\times H_r \to \mathcal{P}_{fc}(L^r)$. We say that the assumptions {\textbf{(F')}} are fulfilled if
\begin{itemize}
  \item[\textbf{(F1)}] $F$ is measurable,
  \item[\textbf{(F2)}] for almost all $t\in(0,T)$, the graph of the mapping $v\mapsto F(t,v)$ is sequentially closed in $H_r\times L^r_w$, and
  \item[\textbf{(F3')}] $\vert{F(t,v)}\vert \leq b(t) \left(1+ \Vert{v}\Vert_{H_r}^{1+\varepsilon}\right)$ a.e. with $b\in L^s(0,T)$, $b\geq 0$ a.e. and $\varepsilon>0$.
\end{itemize}

With these assumptions, we can prove the following result.

\begin{theorem}
Let $\Omega\subset\R^3$ be open, bounded, and connected with $\partial\Omega\in \mathscr{C}^{2,\mu}$, $0<\mu<1$ and let $3<r<\infty$, $1<s<\infty$. Let $F\colon [0,T]\times H_r \to \mathcal{P}_{fc}(L^r)$ satisfy the assumptions \textbf{(F')}. Then, for each $T>0$, there exists an $\alpha_0\in (0,1)$ such that for all $\alpha\in (0,\alpha_0)$ and sufficiently small $u_0\in D^s_r$, $\tau_0\in W^{1,r}$ and $b\in L^s(0,T)$, there exist\footnote{The condition that $\alpha$ has to be chosen small enough means that the influence of Newtonian viscosity on the fluid flow has to be big enough. The Reynolds number Re and the Weissenberg number We can be chosen arbitrarily.}
  \begin{equation*}
    \begin{aligned}
      u&\in L^s(0,T; D(A_r)) \quad\text{with}\quad \partial_t u\in L^s(0,T; H_r),\\
      \tau &\in \mathscr C([0,T];W^{1,r})\quad\text{with}\quad \partial_t \tau\in L^s(0,T; L^r),\\
      p&\in L^s(0,T; W^{1,r}),\\
    \end{aligned}
  \end{equation*}
  such that $(u,\tau,p)$ is a solution to 
  \begin{equation} \label{problem_multivalued_global}
    \left\{\begin{aligned}
      \mathrm{Re} \left(\partial_t u + (u\cdot\nabla) u\right) -(1-\alpha)\Delta u - \nabla\cdot \tau+ \nabla p &\in F(\cdot ,u) \phantom{=02\alpha D(u)\tau_0f}\hspace{-1.5cm} \text{in } (0,T),\\
      \mathrm{We}\left( \partial_t \tau + (u\cdot\nabla)\tau + g_a(\tau,\nabla u)\right) + \tau &= 2\alpha D(u) \phantom{\in 0\tau_0fF(\cdot,u)}\hspace{-1.5cm} \text{in } (0,T),\\
      u(0)=u_0, \quad\quad \tau(0)&=\tau_0.
    \end{aligned}\right.
  \end{equation}    
\end{theorem}

\begin{proof}
  We use the same method as in the proof of Theorem~\ref{thm_local}. Let $\mathcal{Y}(T)$, $\mathcal{X}(T)$ and $\Phi$ be defined as before. As seen in the estimates~\eqref{estimate_R_1_R_2}, for arbitrary $\alpha$, $R_1$, and $R_2$, the initial values $u_0$ and $\tau_0$ can be chosen small enough such that $\mathcal{Y}(T)$ is nonempty. Analogously to the proof before, $\mathcal{Y}(T)$ is also convex and compact and $\Phi$ is well-defined. We now show that 
  \begin{equation*}
    \Phi(\mathcal{Y}(T))\subset \mathcal{Y}(T)
  \end{equation*}
  for $\alpha$, $u_0$, $\tau_0$, and $b$ sufficiently small (and $R_1$, $R_2$, and $R_3$ chosen appropriately). We again proceed similarly to the single-valued case. Let $(u,\tau)\in \Phi(\tilde{u}, \tilde{\tau})$ for an arbitrary $(\tilde{u}, \tilde{\tau})\in \mathcal{Y}(T)$ and let $f\in \mathcal{F}^s(\tilde{u})$ such that $u$ solves the single-valued problem~\eqref{linearised_equation_u_short} with $f$. As before, the a priori estimate for the solution to~\eqref{linearised_equation_u_short} yields (cf. Fern\'andez-Cara, Guill\'en, and Ortega~\cite[p.~575]{FCGO02})
  \begin{equation} \label{estimate_u_global}
    \Vert u\Vert_{L^s(D(A_r))}^s + \Vert \partial_t u\Vert_{L^s(H_r)}^s \leq \left(\frac{c_1}{1-\alpha}\right)^s \left( \Vert u_0\Vert_{D_r^s}^s \, c_5 \left( R_1^s \Vert u_0\Vert_{H_r}^s + R_1^{2s} T^{s-1} + R_2^s T + \Vert f\Vert_{L^s(L^r)}^s \right) \right)
  \end{equation}
  with the same constant $c_1$ as in the proof before and $c_5>0$ depending on Re, $r$, $s$, and $\Omega$. Using Assumption~\textbf{(F3')}, we have 
  \begin{equation*}
    \Vert f\Vert_{L^s(L^r)}^s \leq 2^{s-1} \int_0^T b(t)^s \left( 1 + \Vert \tilde{u}(t)\Vert_{H_r}^{s(1+\varepsilon)}\right) \diff t.
  \end{equation*}
  Again, the estimate~\eqref{estimate_L_infty_H_r} and $(\tilde{u}, \tilde{\tau})\in \mathcal{Y}(T)$ yield
  \begin{equation*}
    \begin{aligned}
      \Vert f\Vert_{L^s(L^r)}^s &\leq 2^{s-1} \Vert b \Vert_{L^s(0,T)}^s\left( 1 +\left( \Vert u_0 \Vert_{H_r} + R_1 T^{1/s'} \right)^{s(1+\varepsilon)}\right) \\
      &\leq 2^{s-1} \Vert b \Vert_{L^s(0,T)}^s \left( 1 +2^{s(1+\varepsilon)-1}\left( \Vert u_0 \Vert_{H_r}^{s(1+\varepsilon)} + R_1^{s(1+\varepsilon)} T^{(s-1)(1+\varepsilon)} \right)\right) \\
      &\leq c_6\;\!\Vert b \Vert_{L^s(0,T)}^s \left( 1 + \Vert u_0 \Vert_{H_r}^{s(1+\varepsilon)} +  R_1^{s(1+\varepsilon)} T^{(s-1)(1+\varepsilon)+1} \right)
    \end{aligned}
  \end{equation*}
  with $c_6=\max(2^{s-1},2^{s(1+\varepsilon)-1})$. Inserted in~\eqref{estimate_u_global}, we obtain
  \begin{equation*}
    \begin{aligned}
      \Vert u\Vert_{L^s(D(A_r))}^s + \Vert \partial_t u\Vert_{L^s(H_r)}^s &\leq \left(\frac{c_1}{1-\alpha}\right)^s \left( \Vert u_0\Vert_{D_r^s}^s + c_5 \left( R_1^s \Vert u_0\Vert_{H_r}^s + R_1^{2s} T^{s-1} + R_2^s T \right.\right.\\
      &\hspace{1cm}\left.\left. +  c_6\;\!\Vert b \Vert_{L^s(0,T)}^s \left( 1 + \Vert u_0 \Vert_{H_r}^{s(1+\varepsilon)} +  R_1^{s(1+\varepsilon)} T^{(s-1)(1+\varepsilon)+1} \right) \right) \right).
    \end{aligned}
  \end{equation*}
  For $\tau$, we have the same estimates as before (cf. \eqref{a_priori_tau} and \eqref{a_priori_dt_tau}), i.e.,
  \begin{equation*}
    \Vert \tau \Vert_{L^\infty(W^{1,r})}\leq \left( \Vert \tau_0 \Vert_{W^{1,r}} + \frac{4\alpha}{c_3\mathrm{We}}\right)\exp\left(c_3\;\!R_1T^{1/s'}\right) = \Lambda
  \end{equation*}
  and 
  \begin{equation*}
    \Vert \partial_t \tau \Vert_{L^s(L^r)}\leq c_4 \;\!\Lambda \left(R_1+ \frac{T^{1/s}}{c_3\mathrm{We}}\right).
  \end{equation*}
  Now, we can choose $\alpha$, $u_0$, $\tau_0$, and $b$ sufficiently small as well as $R_1$, $R_2$, and $R_3$ in such a way that the right-hand side of the last three inequalities can be estimated by $R_1^s$, $R_2$, and $R_3$, respectively: First, let $\alpha_1\in(0,1)$ be arbitrary. Now, choose $R_1$ small enough such that
  \begin{equation*}
    \left(\frac{c_1}{1-\alpha_1}\right)^s c_5 \left( R_1^{2s}T^{s-1} + c_6\,R_1^{s(1+\varepsilon)} T^{(s-1)(1+\varepsilon)+1} \right) < R_1^s,
  \end{equation*}
  e.g.,
  \begin{equation*}
    R_1 < \frac{1}{2}\min\left( \left( \left(\frac{c_1}{1-\alpha_1}\right)^{s} c_5 T^{s-1}\right)^{-1/s},\; \left(\left( \frac{c_1}{1-\alpha_1} \right)^{s} c_5c_6 \, T^{(s-1)(1+\varepsilon)+1}\right)^{-1/(s\varepsilon)} \right).
  \end{equation*}
  Next, choose $R_2$ small enough such that 
  \begin{equation*}
    \left(\frac{c_1}{1-\alpha_1}\right)^s c_5 \left( R_1^{2s}T^{s-1} + R_2^s T + c_6\,R_1^{s(1+\varepsilon)} T^{(s-1)(1+\varepsilon)+1} \right) < R_1^s,
  \end{equation*}
  choose $\alpha_0\in (0,\alpha_1]$ small enough such that
  \begin{equation*}
    \frac{4\alpha_0}{c_3\mathrm{We}}\exp\left(c_3\;\!R_1T^{1/s'}\right) < R_2,
  \end{equation*}
  and choose $R_3$ big enough such that
  \begin{equation*}
    c_4 \left(R_1+ \frac{T^{1/s}}{c_3\mathrm{We}}\right)\frac{4\alpha_0}{c_3\mathrm{We}}\exp\left(c_3\;\!R_1T^{1/s'}\right) < R_3.
  \end{equation*}
  Finally, choose $u_0$, $\tau_0$ and $b$ sufficiently small such that
  \begin{equation*}
    \begin{aligned}
      \left(\frac{c_1}{1-\alpha}\right)^s \left( \Vert u_0\Vert_{D_r^s}^s + c_5 \left( R_1^s \Vert u_0\Vert_{H_r}^s + R_1^{2s} T^{s-1} + R_2^s T \right.\right.\hspace{2cm}&\\
      \left.\left. +  c_6\;\!\Vert b \Vert_{L^s(0,T)}^s \left( 1 + \Vert u_0 \Vert_{H_r}^{s(1+\varepsilon)} +  R_1^{s(1+\varepsilon)} T^{(s-1)(1+\varepsilon)+1} \right) \right) \right) &\leq R_1^s,\\
      \left( \Vert \tau_0 \Vert_{W^{1,r}} + \frac{4\alpha}{c_3\mathrm{We}}\right)\exp\left(c_3\;\!R_1T^{1/s'}\right) &\leq R_2,\\
      c_4 \;\!\Lambda \left(R_1+ \frac{T^{1/s}}{c_3\mathrm{We}}\right) &\leq R_3,
    \end{aligned}
  \end{equation*}
  and thus $\Phi(\mathcal{Y}(T))\subset \mathcal{Y}(T)$. Finally, we can show that $\Phi$ is closed analogously to the proof before. Thus, we can now apply the generalisation of Kakutani's fixed-point theorem (see Glicksberg~\cite{Glicksberg} and Fan~\cite{Fan}) to obtain the existence of a fixed point and therefore a solution to problem~\eqref{problem_multivalued_global}.\qed
\end{proof}



\end{document}